\documentclass{amsart}
\usepackage[all]{xy}
\usepackage{color}
\usepackage{amsthm}
\usepackage{amssymb}
\usepackage[colorlinks=true]{hyperref}
\numberwithin{equation}{section}

\newtheorem{theorem}[equation]{Theorem}
\newtheorem*{theorem*}{Theorem}
\newtheorem{lemma}[equation]{Lemma}

\newtheorem*{conjecture*}{Mamma Conjecture}
\newtheorem*{conjecture1*}{Mamma Conjecture (revisited)}
\newtheorem{proposition}[equation]{Proposition}

\newtheorem*{corollary*}{Corollary}

\theoremstyle{remark}
\newtheorem{definition}[equation]{Definition}

\newtheorem{notation}[equation]{Notation}

\theoremstyle{remark}
\newtheorem{remark}[equation]{Remark}

\newcommand{\cA}{{\mathcal A}}
\newcommand{\cB}{{\mathcal B}}
\newcommand{\cC}{{\mathcal C}}
\newcommand{\cD}{{\mathcal D}}

\newcommand{\cK}{{\mathcal K}}

\newcommand{\cN}{{\mathcal N}}

\newcommand{\cT}{{\mathcal T}}

\newcommand{\dgHo}{\mathsf{H}^0}

\newcommand{\pe}{\mathsf{pe}}

\newcommand{\bbQ}{\mathbb{Q}}
\newcommand{\bbZ}{\mathbb{Z}}

\DeclareMathOperator{\id}{id}

\DeclareMathOperator{\NChow}{NChow} % category of noncommutative Chow motives
\DeclareMathOperator{\NNum}{NNum} % category of noncommutative numerical motives
\DeclareMathOperator{\NNumK}{NC_{\mathrm{num}}} % category of noncommutative numerical motives (Kontsevich)
\newcommand{\dg}{\mathsf{dg}}

\newcommand{\Hom}{\mathsf{Hom}}
\newcommand{\End}{\mathsf{End}}

\newcommand{\op}{\mathsf{op}}

\newcommand{\too}{\longrightarrow}

\newcommand{\ie}{\textsl{i.e.}\ }

\title[Kontsevich's noncommutative numerical motives]{Kontsevich's noncommutative numerical motives}
\author{Matilde Marcolli and Gon{\c c}alo~Tabuada}

\address{Matilde Marcolli, Mathematics Department, Mail Code 253-37, Caltech, 1200 E.~California Blvd. Pasadena, CA 91125, USA}
\email{matilde@caltech.edu} 

\address{Gon{\c c}alo Tabuada, Department of Mathematics, MIT, Cambridge, MA 02139}
\email{tabuada@math.mit.edu}

\subjclass[2000]{18D20, 18F30, 18G55, 19A49, 19D55}
\date{\today}

\keywords{Noncommutative algebraic geometry, noncommutative motives}

\thanks{The first named author was partially supported by the NSF grants DMS-0901221 and DMS-1007207.}

\begin{document}
\begin{abstract}
In this note we prove that Kontsevich's category $\NNumK(k)_F$ of noncommutative numerical motives is equivalent to the one constructed by the authors in \cite{Semi}. As a consequence, we conclude that $\NNumK(k)_F$ is abelian semi-simple as conjectured by Kontsevich.
\end{abstract}

\maketitle
\vskip-\baselineskip
\vskip-\baselineskip
%\vskip-\baselineskip
%-------------------------------------------------------------------
\section{Introduction and statement of results}\label{Intro}
%-------------------------------------------------------------------
Over the past two decades Bondal, Drinfeld, Kaledin, Kapranov, Kontsevich, Van den Bergh, and others, have been promoting a broad noncommutative (algebraic) geometry program where ``geometry'' is performed directly on dg categories; see~\cite{Kapranov1,Kapranov,BV,Drinfeld,Chitalk,Kaledin,IAS,ENS,Miami,finMot}. Among many developments, Kontsevich introduced a rigid symmetric monoidal category $\NNumK(k)_F$ of noncommutative numerical motives (over a ground field $k$ and with coefficients in a field $F$); consult \S\ref{sec:K} for details. The key ingredient in his approach is the existence of a well-behaved bilinear form on the Grothendieck group of certain smooth and proper dg categories.

Recently, the authors introduced in \cite{Semi} an alternative rigid symmetric monoidal category $\NNum(k)_F$ of noncommutative numerical motives; consult \S\ref{sec:semi}. In contrast with Kontsevich's approach, the authors used Hochschild homology in order to formalize the word ``counting'' in the noncommutative world. Our main result is the following\,:
\begin{theorem}\label{thm:main}
The categories $\NNumK(k)_F$ and $\NNum(k)_F$ are equivalent.
\end{theorem}
By combining Theorem~\ref{thm:main} with \cite[Thm.~1.9]{Semi} we then obtain\,:
\begin{theorem}\label{thm:main2}
Assume that $k$ is a field extension of $F$ or vice-versa. Then, the category $\NNumK(k)_F$ is abelian semi-simple.
\end{theorem}
Assuming several (polarization) conjectures, Kontsevich conjectured Theorem~\ref{thm:main2} in the particular case where $F=\bbQ$ and $k$ is of characteristic zero; see \cite{IAS}. We observe that Kontsevich's beautiful insight not only holds much more generally, but moreover it does not require the assumption of any (polarization) conjecture.

%-------------------------------------------------------------------
\subsection*{Notations}\label{sec:notations}
%-------------------------------------------------------------------
We will work over a (fixed) ground field $k$. The field of coefficients will be denoted by $F$. Let $(\cC(k),\otimes,k)$ be the symmetric monoidal category of complexes of $k$-vector spaces. We will use {\em cohomological} notation, \ie the differential increases the degree.
%-------------------------------------------------------------------
\section{Differential graded categories}\label{sec:dg}
%-------------------------------------------------------------------
A {\em differential graded (=dg) category} $\cA$ (over $k$) is a category enriched over $\cC(k)$, \ie the morphism sets $\cA(x,y)$ are complexes of $k$-vector spaces and the composition operation fulfills the Leibniz rule $d(f \circ g)=d(f) \circ g + (-1)^{\mathrm{deg}(f)}\circ d(g)$; consult Keller's ICM address~\cite{ICM} for further details.

The {\em opposite} dg category $\cA^\op$ has the same objects as $\cA$ and complexes of morphisms given by $\cA^\op(x,y):=\cA(y,x)$. The $k$-linear category $\dgHo(\cA)$ has the same objects as $\cA$ and morphisms given by $\dgHo(\cA)(x,y):= \textrm{H}^0\cA(x,y)$, where $\mathrm{H}^0$ denotes $0^{\mathrm{th}}$-cohomology. A {\em right dg $\cA$-module} $M$ (or simply a $\cA$-module) is a dg functor $M:\cA^\op \to \cC_\dg(k)$ with values in the dg category $\cC_\dg(k)$ of complexes of $k$-vector spaces. We will denote by $\cC(\cA)$ the category of $\cA$-modules. Recall from \cite[\S3]{ICM} that $\cC(\cA)$ carries a {\em projective} model structure. Moreover, the differential graded structure of $\cC_\dg(k)$ makes $\cC(\cA)$ naturally into a dg category $\cC_\dg(\cA)$. The dg category $\cC_\dg(\cA)$ endowed with the projective model structure is a {\em $\cC(k)$-model category} in the sense of \cite[Def.~4.2.18]{Hovey}. Let $\cD(\cA)$ be the {\em derived category} of $\cA$, \ie the localization of $\cC(\cA)$ with respect to the class of weak equivalences. Its full triangulated subcategory of compact objects (\ie those $\cA$-modules $M$ such that the functor $\Hom_{\cD(\cA)}(M,-)$ preserves arbitrary sums; see \cite[Def.~4.2.7]{Neeman}) will be denoted by $\cD_c(\cA)$.
\begin{notation}\label{not:perf}
We will denote by $\widehat{\cA}_\pe$ the full dg subcategory of $\cC_\dg(\cA)$ consisting of those cofibrant $\cA$-modules which become compact in $\cD(\cA)$.
Since all the objects in $\cC(\cA)$ are fibrant, and $\cC_\dg(\cA)$ is a $\cC(k)$-model category, we have natural isomorphisms of $k$-vector spaces
\begin{equation}\label{eq:nat-iso}
\textrm{H}^i\widehat{\cA}_\pe(M,N) \simeq \Hom_{\cD_c(\cA)}(M,N[-i]) \qquad i \in \bbZ\,.
\end{equation}
As any $\cA$-module admits a (functorial) cofibrant approximation, we obtain a natural equivalence of triangulated categories $\dgHo(\widehat{\cA}_\pe)\simeq \cD_c(\cA)$.
\end{notation}
The {\em tensor product} $\cA\otimes \cB$ of two dg categories is defined as follows: the set of objects is the cartesian product of the sets of objects, and the complexes of morphisms are given by $(\cA\otimes \cB)((x,x'),(y,y')):= \cA(x,y) \otimes \cB(x',y')$. A {\em $\cA\textrm{-}\cB$-bimodule $X$} is a dg functor $X: \cA\otimes \cB^\op \to \cC_\dg(k)$, or in other words a $(\cA^\op \otimes \cB)$-module.

\begin{definition}[Kontsevich \cite{IAS,ENS}]
A dg category $\cA$ is {\em smooth} if the $\cA\textrm{-}\cA$-bimodule 
\begin{eqnarray*}%\label{eq:diagonal}
\cA(-,-): \cA\otimes \cA^\op \too \cC_\dg(k) && (x,y) \mapsto \cA(y,x)
\end{eqnarray*}
belongs to $\cD_c(\cA^\op\otimes \cA)$, and {\em proper} if for each ordered pair of objects $(x,y)$ we have $ \sum_i \mathrm{dim}\, \textrm{H}^i\cA(x,y) < \infty$.
\end{definition}

%-------------------------------------------------------------------
\section{Noncommutative Chow motives}
%-------------------------------------------------------------------
The rigid symmetric monoidal category $\NChow(k)_F$ of {\em noncommutative Chow motives} was constructed\footnote{In {\em loc.~cit.} we have worked more generally over a ground commutative ring $k$.} in \cite{CvsNC, IMRN}. It is defined as the pseudo-abelian envelope of the category whose objects are the smooth and proper dg categories, whose morphisms from $\cA$ to $\cB$ are given by the $F$-linearized Grothendieck group $K_0(\cA^\op \otimes \cB)_F$, and whose composition operation is induced by the tensor product of bimodules. In analogy with the commutative world, the morphisms of $\NChow(k)_F$ are called {\em correspondences}. The symmetric monoidal structure is induced by the tensor product of dg categories. 
%-------------------------------------------------------------------
\section{Kontsevich's approach}\label{sec:K}
%-------------------------------------------------------------------
 In this section we recall and enhance Kontsevich's construction of the category $\NNumK(k)_F$ of noncommutative numerical motives; consult \cite{IAS}. Let $\cA$ be a proper dg category. By construction, the dg category $\widehat{\cA}_\pe$ is also proper and we have a natural equivalence of triangulated categories $\dgHo(\widehat{\cA}_\pe)\simeq \cD_c(\cA)$. Hence, thanks to the natural isomorphisms \eqref{eq:nat-iso}, we can consider the following assignment
\begin{eqnarray*}%\label{eq:assignment}
 \mathrm{obj}\, \cD_c(\cA) \times \mathrm{obj} \,\cD_c(\cA) \too \bbZ && (M,N) \mapsto \chi(M,N)\,,
\end{eqnarray*} 
where $\chi(M,N)$ is the integer 
$$\sum_i (-1)^i \mathrm{dim}\, \Hom_{\cD_c(\cA)}(M,N[-i])\,.$$ 

Recall that the Grothendieck group $K_0(\cA)$ of $\cA$ can be defined as the Grothendieck group of the triangulated category $\cD_c(\cA)$. A simple verification shows that the above assignment gives rise to a well-defined bilinear form $K_0(\cA) \otimes_\bbZ K_0(\cA) \to \bbZ$.
 By tensoring it with $F$, we then obtain
\begin{equation}\label{eq:bilinear}
\chi(-,-): K_0(\cA)_F \otimes_F K_0(\cA)_F \too F\,.
\end{equation}
The bilinear form \eqref{eq:bilinear} is in general not symmetric. Let 
$$\mathrm{Ker}_L(\chi):=\{\underline{M} \in K_0(\cA)_F \, |\, \chi(\underline{M},\underline{N})=0\,\, \mathrm{for}\,\, \mathrm{all}\,\, \underline{N} \in K_0(\cA)_F\}$$
$$\mathrm{Ker}_R(\chi):=\{\underline{N} \in K_0(\cA)_F \, |\, \chi(\underline{M},\underline{N})=0\,\, \mathrm{for}\,\, \mathrm{all}\,\, \underline{M} \in K_0(\cA)_F\}$$
be, respectively, its left and right kernel. These $F$-linear subspaces of $K_0(\cA)_F$ are in general distinct. However, as we will prove in Theorem~\ref{thm:K1}, they agree when we assume that $\cA$ is moreover smooth. In order to prove this result, let us start by recalling Bondal-Kapranov's notion of a Serre functor. Let $\cT$ be a $k$-linear {\em $\mathrm{Ext}$-finite} triangulated category, \ie $\sum_i \mathrm{dim}\, \Hom_\cT(M,N[-i])< \infty$ for any two objects $M$ and $N$ in $\cT$. Following Bondal and Kapranov \cite[\S3]{Kapranov1}, a {\em Serre functor} $S: \cT \stackrel{\sim}{\to} \cT$ is an autoequivalence together with bifunctorial isomorphisms
\begin{equation}\label{eq:Serre}
\Hom_\cT(M,N) \simeq \Hom_\cT(N,S(M))^\ast\,,
\end{equation}
where $(-)^\ast$ stands for the $k$-duality functor. Whenever a Serre functor exists, it is unique up to isomorphism.
\begin{theorem}\label{thm:aux}
Let $\cA$ be a smooth and proper dg category. Then, the triangulated category $\cD_c(\cA)$ admits a Serre functor.
\end{theorem}
\begin{proof}
Note first that the properness of $\widehat{\cA}_\pe$, the equivalence of categories $\dgHo(\widehat{\cA}_\pe)\simeq \cD_c(\cA)$, and the natural isomorphisms \eqref{eq:nat-iso}, imply that $\cD_c(\cA)$ is $\mathrm{Ext}$-finite. By combining \cite[Corollary~3.5]{Kapranov1} with \cite[Thm.~1.3]{BV}, it suffices then to show that $\cD_c(\cA)$ is pseudo-abelian and that it admits a strong generator; consult \cite[page~2]{BV} for the notion of {\em strong} generator. The fact that $\cD_c(\cA)$ is pseudo-abelian is clear from its own definition. In order to prove that it admits a strong generator, we may combine \cite[Prop.~4.10]{CT1} with \cite[Thm.~4.12]{ICM} to conclude that $\cA$ is dg Morita equivalent to a dg algebra $A$. Hence, without loss of generality, we may replace $\cA$ by $A$. The proof that $\cD_c(A)$ admits a strong generator now follows from the arguments of Shklyarov on \cite[page~7]{Shklyarov}, which were inspired by Bondal-Van den Bergh's original proof of \cite[Thm.~3.1.4]{BV}.
\end{proof}
\begin{lemma}\label{lem:aux1}
Let $\cA$ be a smooth and proper dg category and $M, N \in \cD_c(\cA)$. Then, we have the following equalities
$$ \chi(M,N) =\chi(N,S(M))=\chi(S^{-1}(N),M)\,,$$
where $S$ is the Serre functor given by Theorem~\ref{thm:aux}.
\end{lemma}
\begin{proof}
Consider the following sequence of equalities\,:
\begin{eqnarray}
\chi(M,N) & =& \sum_i (-1)^i \mathrm{dim}\, \Hom_{\cD_c(\cA)}(M,N[-i]) \nonumber \\
& =& \sum_i (-1)^i \mathrm{dim}\, \Hom_{\cD_c(\cA)}(N[-i],S(M)) \label{eq:equality3} \\
& =& \sum_i (-1)^i \mathrm{dim}\, \Hom_{\cD_c(\cA)}(N,S(M)[i]) \label{eq:equality4} \\
& =& \chi(N,S(M))\,. \label{eq:equality5}
\end{eqnarray}
Equivalence \eqref{eq:equality3} follows from the bifunctorial isomorphisms \eqref{eq:Serre} and from the fact that a finite dimensional $k$-vector space and its $k$-dual have the same dimension. Equivalence \eqref{eq:equality4} follows from the fact that the suspension functor in an autoequivalence of the triangulated category $\cD_c(\cA)$. Finally, equivalence~\eqref{eq:equality5} follows from a reordering of the finite sum which does not alter the sign of each term. This shows the equality $\chi(M,N)=\chi(N,S(M))$. The equality $\chi(M,N)=\chi(S^{-1}(N),M)$ is proven in  a similar way. Simply use
$$ \Hom_\cT(M,N) \simeq \Hom_\cT(S^{-1}(N),M)^\ast$$
instead of the bifunctorial isomorphisms \eqref{eq:Serre}.
\end{proof}
\begin{theorem}\label{thm:K1}
Let $\cA$ be a smooth and proper dg category. Then, $\mathrm{Ker}_L(\chi)=\mathrm{Ker}_R(\chi)$; the resulting well-defined subspace of $K_0(\cA)_F$ will be denoted by $\mathrm{Ker}(\chi)$.
\end{theorem}
\begin{proof}
We start by proving the inclusion $\mathrm{Ker}_L(\chi)\subseteq \mathrm{Ker}_R(\chi)$. Let $\underline{M}$ be an element of $\mathrm{Ker}_L(\chi)$. Since $K_0(\cA)_F$ is generated by the elements of shape $[N]$, with $N \in \cD_c(\cA)$, it suffices then to show that $\chi([N],\underline{M})=0$ for every such $N$. Note that $\underline{M}$ can be written as $[a_1 M_1 + \cdots  + a_n M_n]$, with $a_1, \ldots , a_n \in F$ and $M_1, \ldots, M_n \in \cD_c(\cA)$. We have then the following equalities
\begin{eqnarray}
\chi([N],\underline{M}) & =& a_1\chi(N,M_1) + \cdots + a_n \chi(N,M_n) \nonumber \\
& =& a_1\chi(M_1,S(N)) + \cdots + a_n \chi(M_n,S(N)) \label{eq:equality6} \\
& =& \chi(\underline{M},[S(N)])\,, \nonumber
\end{eqnarray}
where \eqref{eq:equality6} follows from Lemma~\ref{lem:aux1}. Finally, since by hypothesis $\underline{M}$ belongs to $\mathrm{Ker}_L(\chi)$, we have $\chi(\underline{M},[S(N)])=0$ and so we conclude that $\chi([N],\underline{M})=0$. Using the equality $\chi(M,N)=\chi(S^{-1}(N),M)$ of Lemma~\ref{lem:aux1}, the proof of the inclusion $\mathrm{Ker}_R(\chi) \subseteq \mathrm{Ker}_L(\chi)$ is similar.
\end{proof}
Let $(\cA,e)$ and $(\cB,e')$ be two noncommutative Chow motives. Recall that $\cA$ and $\cB$ are smooth and proper dg categories and that $e$ and $e'$ are idempotent elements of $K_0(\cA^\op \otimes \cA)_F$ and $K_0(\cB^\op \otimes \cB)_F$, respectively. Recall also that
\begin{equation}\label{eq:description}
\Hom_{\NChow(k)_F}((\cA,e),(\cB,e')):= (e \circ K_0(\cA^\op \otimes \cB)_F \circ e')\,.
\end{equation}
Since smooth and proper dg categories are stable under tensor product (see \cite[\S4]{CT1}), the above bilinear form \eqref{eq:bilinear} (applied to $\cA=\cA^\op \otimes \cB$) restricts to a bilinear form
\begin{equation*}%\label{eq:main1}
\chi(-,-): \Hom_{\NChow(k)_F}((\cA,e),(\cB,e')) \otimes_F \Hom_{\NChow(k)_F}((\cA,e),(\cB,e')) \too F\,.
\end{equation*}
By Theorem~\ref{thm:K1} we obtain then a well-defined kernel $\mathrm{Ker}(\chi)$. These kernels (one for each ordered pair of noncommutative Chow motives) assemble themselves in a $\otimes$-ideal $\cK\mathrm{er}(\chi)$ of the category $\NChow(k)_F$.
\begin{definition}[Kontsevich~\cite{IAS}]
The category $\NNumK(k)_F$ of {\em noncommutative numerical motives} (over $k$ and with coefficients in $F$) is the pseudo-abelian envelope of the quotient category $\NChow(k)_F/\cK\mathrm{er}(\chi)$.
\end{definition}
\begin{remark}
The fact that $\cK\mathrm{er}(\chi)$ is a well-defined $\otimes$-ideal of $\NChow(k)_F$ will become clear(er) after the proof of Theorem~\ref{thm:main}.
\end{remark}
%-------------------------------------------------------------------
\section{Alternative approach}\label{sec:semi}
%-------------------------------------------------------------------
The authors introduced in \cite{Semi} an alternative category $\NNum(k)_F$ of noncommutative numerical motives. Let $(\cA,e)$ and $(\cB,e')$ be two noncommutative Chow motives and $\underline{X}=(e \circ [\sum_i a_iX_i]\circ e')$ and $\underline{Y}=(e' \circ [\sum_j b_jY_j]\circ e)$ 
two correspondences. Recall that $X_i$ and $Y_j$ are bimodules and that the sums are indexed by a finite set. The {\em intersection number} $ \langle \underline{X} \cdot \underline{Y} \rangle$ of $\underline{X}$ with $\underline{Y}$ is given by the formula
\begin{equation*}%\label{eq:intersection}
\sum_{i,j,n} (-1)^n\, a_i \cdot b_j \cdot \mathrm{dim} \, HH_n(\cA,X_i \otimes_\cB Y_j) \in F\,,
\end{equation*}
where $HH_n(\cA,X_i\otimes_\cB Y_j)$ denotes the $n^{\mathrm{th}}$-Hochschild homology group of $\cA$ with coefficients in the $\cA\textrm{-}\cA$-bimodule $X_i \otimes_\cB Y_j$. This procedure gives rise to a well-defined bilinear pairing
$$\langle -\cdot -\rangle : \Hom_{\NChow(k)_F}((\cA,e),(\cB,e')) \otimes_F \Hom_{\NChow(k)_F}((\cB,e'),(\cA,e)) \too F\,.$$
In contrast with $\chi(-,-)$, this bilinear pairing is symmetric. A correspondence $\underline{X}$ is {\em numerically equivalent to zero} if for every correspondence $\underline{Y}$ the intersection number $\langle \underline{X} \cdot \underline{Y}\rangle$ is zero. As proved in \cite[Thm.~1.5]{Semi}, the correspondences which are numerically equivalent to zero form a $\otimes$-ideal $\cN$ of the category $\NChow(k)_F$. The {\em category of noncommutative numerical motives} $\NNum(k)_F$ is then defined as the pseudo-abelian envelope of the quotient category $\NChow(k)_F/\cN$.
%-------------------------------------------------------------------
\section{Proof of Theorem~\ref{thm:main}}
%-------------------------------------------------------------------
The proof will consist on showing that the $\otimes$-ideals $\cK\mathrm{er}(\chi)$ and $\cN$, described respectively in \S\ref{sec:K} and \S\ref{sec:semi}, are exactly the same. As explained in the proof of Theorem~\ref{thm:aux} it is equivalent to work with smooth and proper dg categories or with smooth and proper dg algebras. In what follows we will use the latter approach. 

Let $A$ be a dg algebra and $M$ a right dg $A$-module. We will denote by $D(M)$ its {\em dual}, \ie the left dg $A$-module $\cC_\dg(A)(M,A)$.
This procedure is (contravariantly) functorial in $M$, and thus gives rise to a triangulated functor $\cD(A) \to \cD(A^\op)^\op$ which restricts to an equivalence $\cD_c(A) \stackrel{\sim}{\to} \cD_c(A^\op)^\op$. Since the Grothendieck group of a triangulated category is canonically isomorphic to the one of the opposite category, we obtain then an induced isomorphism $K_0(A)_F \stackrel{\sim}{\to} K_0(A^\op)_F$.
\begin{proposition}\label{prop:trace}
Let $A$ and $B$ be two smooth and proper dg algebras and $X,Y \in \cD_c(A^\op \otimes B)$. Then, $\chi(X,Y) \in F$ agrees with the categorical trace of the correspondence $[Y \otimes_B D(X)] \in \End_{\NChow(k)_F}((A,\id_A))$. 
\end{proposition}
\begin{proof}
The $A\textrm{-}B$-bimodules $X$ and $Y$ give rise, respectively, to correspondences $[X]: (A, \id_A) \to (B, \id_B)$ and $[Y]: (A, \id_A) \to (B, \id_B)$ in $\NChow(k)_F$. On the other hand, the $B\textrm{-}A$-bimodule $D(X):=\widehat{(A^\op \otimes B)}_\pe(X,A^\op \otimes B) \in \cD_c(B^\op \otimes A)$ (see Notation~\ref{not:perf}) gives rise to a correspondence $[D(X)]: (B,\id_B) \to (A,\id_A)$. We can then consider the following composition
\begin{equation}\label{eq:endo}
[Y \otimes_B D(X)]: (A,\id_A) \stackrel{[Y]}{\too} (B,\id_B) \stackrel{[D(X)]}{\too} (A,\id_A)\,.
\end{equation}
Recall from \cite{CvsNC} that the $\otimes$-unit of $\NChow(k)_F$ is the noncommutative motive $(k,\id_k)$, where $k$ is the ground field considered as a dg algebra concentrated in degree zero. Recall also that the dual of $(A, \id_A)$ is $(A^\op, \id_{A^\op})$ and that the evaluation map $(A, \id_A) \otimes (A^\op, \id_{A^\op}) \stackrel{\mathrm{ev}}{\to} (k, \id_k)$ is given by the class in $K_0(A^\op \otimes A)_F$ of $A$ considered as a $A\textrm{-}A$-bimodule.
Hence, the categorical trace of the correspondence \eqref{eq:endo} is the following composition
$$ (k, \id_k) \stackrel{[Y \otimes_B D(X)]}{\too} (A^\op,\id_{A^\op}) \otimes (A,\id_A) \simeq (A,\id_A) \otimes (A^\op,\id_{A^\op}) \stackrel{[A]}{\too} (k, \id_k)\,.$$
Since the composition operation in $\NChow(k)_F$ is given by the tensor product of bimodules, the above composition corresponds to the class in $K_0(k)_F \simeq F$ of the complex of $k$-vector spaces
\begin{equation}\label{eq:final}
(Y \otimes_B D(X)) \otimes_{A^\op \otimes A} A^\op\,.
\end{equation}
Thanks to the natural isomorphisms 
\begin{equation*}
(Y \otimes_B D(X))\otimes_{A^\op \otimes A} A^\op \simeq  Y \otimes_{A^\op \otimes B} D(X) \simeq \widehat{(A^\op \otimes B)}_\pe(X,Y)
\end{equation*}
we conclude that \eqref{eq:final} is naturally isomorphic to $\widehat{(A^\op \otimes B)}_\pe(X,Y)$. As a consequence they have the same Euler characteristic 
\begin{equation*}%\label{eq:Eulereq}
\sum_i(-1)^i \mathrm{dim}\,\textrm{H}^i((Y \otimes_B D(X)) \otimes_{A^\op \otimes A} A^\op) = \sum_i(-1)^i \mathrm{dim}\,\textrm{H}^i(\widehat{(A^\op \otimes B)}_\pe(X,Y))\,.
\end{equation*}
The natural isomorphisms of $k$-vector spaces \eqref{eq:nat-iso} (applied to $\cA=A^\op \otimes B$, $M=X$ and $N=Y$) allow us then to conclude that the right hand-side of the above equality agrees with $\chi(X,Y) \in \bbZ$. On the other hand, the left hand-side is simply the class of the complex \eqref{eq:final} in the Grothendieck group $K_0(k)=\bbZ$. As a consequence, this equality holds also on the $F$-linearized Grothendieck group $K_0(k)_F\simeq F$ and so the proof is finished. 
\end{proof}
Now, let $(A,e)$ and $(B,e')$ be two noncommutative Chow motives (with $A$ and $B$ dg algebras). As explained above, the duality functor induces an isomorphism $K_0(A^\op \otimes B)_F \simeq K_0(B^\op \otimes A)_F$ on the $F$-linearized Grothendieck groups. Via the description \eqref{eq:description} of the Hom-sets of $\NChow(k)_F$, we obtain then an induced duality isomorphism
\begin{equation}\label{eq:res-iso}
D(-): \Hom_{\NChow(k)_F}((A,e),(B,e')) \stackrel{\sim}{\too} \Hom_{\NChow(k)_F}((B,e'),(A,e))\,.
\end{equation} 
\begin{proposition}\label{prop:diag-com}
The following square
$$
\xymatrix{
\Hom_{\NChow(k)_F}((A,e),(B,e')) \otimes_F \Hom_{\NChow(k)_F}((A,e),(B,e'))  \ar[d]_{\eqref{eq:res-iso}\otimes \id}^\simeq \ar[rr]^-{\chi(-,-)} && F \ar@{=}[d] \\
\Hom_{\NChow(k)_F}((B,e'),(A,e))  \otimes_F \Hom_{\NChow(k)_F}((A,e),(B,e'))  \ar[rr]_-{\langle - \cdot -\rangle} && F
}
$$
is commutative.
\end{proposition}
\begin{proof}
Since the $F$-linearized Grothendieck group $K_0(A^\op \otimes B)_F$ is generated by the elements of shape $[X]$, with $X \in \cD_c(A^\op \otimes B)$, and $\chi(-,-)$ and $\langle -\cdot - \rangle$ are bilinear, it suffices to show the commutativity of the above square with respect to the correspondences $\underline{X}=(e \circ [X]\circ e')$ and $\underline{Y}=(e \circ [Y] \circ e')$. By Proposition~\ref{prop:trace}, $\chi(\underline{X},\underline{Y})=\chi(X,Y) \in F$ agrees with the categorical trace in $\NChow(k)_F$ of the correspondence $[Y\otimes_B D(X)] \in \End_{\NChow(k)_F}((A,\id_A))$. 

On the other hand, since the bilinear pairing $\langle -\cdot - \rangle$ is symmetric, we have the following equality $\langle D(\underline{X})\cdot \underline{Y}\rangle = \langle \underline{Y} \cdot D(\underline{X})\rangle$. By \cite[Corollary~4.4]{Semi}, we then conclude that the intersection number $\langle \underline{Y} \cdot D(\underline{X}) \rangle$ agrees also with the categorical trace of the correspondence $[Y\otimes_B D(X)]$. The proof is then achieved.
\end{proof}
We now have all the ingredients needed to prove Theorem~\ref{thm:main}. We will show that a correspondence $\underline{X} \in \Hom_{\NChow(k)_F}((A,e),(B,e'))$ belongs to $\mathrm{Ker}(\chi)$ if and only if it is numerically equivalent to zero. Assume first that $\underline{X} \in \mathrm{Ker}_R(\chi)=\mathrm{Ker}(\chi)$. Then, by Proposition~\ref{prop:diag-com}, the intersection number $\langle D(\underline{Y}) \cdot \underline{X} \rangle$ is trivial for every correspondence $\underline{Y} \in \Hom_{\NChow(k)_F}((A,\id_A), (B,\id_B))$. The symmetry of the bilinear pairing $\langle -\cdot -\rangle$, combined with isomorphism \eqref{eq:res-iso}, allow us then to conclude that $\underline{X}$ is numerically equivalent to zero. 

Now, assume that $\underline{X}$ is numerically equivalent to zero. Once again the symmetry of the bilinear pairing $\langle-\cdot -\rangle$, combined with isomorphism \eqref{eq:res-iso}, implies that $\chi(\underline{Y},\underline{X})=0$ for every correspondence $\underline{Y} \in \Hom_{\NChow(k)_F}((A,\id_A),(B,\id_B))$. As a consequence, $\underline{X} \in \mathrm{Ker}_R(\chi)=\mathrm{Ker}(\chi)$. The above arguments hold for all noncommutative Chow motives and correspondences. Therefore, the $\otimes$-ideals $\cK\mathrm{er}(\chi)$ and $\cN$, described respectively in \S\ref{sec:K} and \S\ref{sec:semi}, are exactly the same and so the proof of Theorem~\ref{thm:main} is finished. 

\medbreak

\noindent\textbf{Acknowledgments:} The authors are very grateful to Yuri~Manin for stimulating discussions.

\end{document}